\documentclass[12pt,oneside,reqno]{amsart}
\usepackage{amssymb}
\usepackage[active]{srcltx}
\usepackage{a4wide}
\usepackage{amsmath,amsthm}
\usepackage{amssymb}
\usepackage{amstext}
\everymath{\displaystyle}
\usepackage{cite}
\usepackage{epsfig}
\usepackage{enumerate}
\usepackage{color}
\usepackage{amssymb}
\newtheorem{theorem}{Theorem}[section]
\newtheorem{lemma}[theorem]{Lemma}
\newtheorem{proposition}[theorem]{Proposition}
\theoremstyle{definition}
\newtheorem{definition}[theorem]{Definition}

\theoremstyle{remark}
\newtheorem{remark}[theorem]{Remark}

\usepackage[linkcolor=blue, urlcolor=blue, citecolor=blue,
colorlinks, bookmarks]{hyperref}

\newcommand{\be}{\begin{equation}}
\newcommand{\ee}{\end{equation}}

\begin{document}

\title[]{ Fractal polynomials On the Sierpi\'nski gasket and some dimensional results}



\author{V. Agrawal}
\address{Department of Mathematical Sciences, IIT(BHU), Varanasi, India 221005 }
\email{vishal.agrawal1992@gmail.com}

\author{S. Verma}
\address{Department of Applied Sciences, IIIT Allahabad, Prayagraj, India 211015}
\email{saurabhverma@iiita.ac.in}
\author{T. Som}
\address{Department of Mathematical Sciences, IIT(BHU), Varanasi, India 221005}
\email{tsom.apm@iitbhu.ac.in}





 

\begin{abstract}
In this paper, we explore some significant properties associated with a fractal operator on the space of all continuous functions defined on the Sierpi\'nski Gasket (SG). We also provide some results related to constrained approximation with fractal polynomials and study the best approximation properties of fractal polynomials defined on the SG. Further we discuss some remarks on the class of polynomials defined on the SG and try to estimate the fractal dimensions of the graph of $\alpha$- fractal function defined on the SG by using the oscillation of functions.
\end{abstract}

\maketitle



\section{Introduction} 
The foundational framework of  Fractal interpolation functions (FIFs)  was first devised by Barnsley \cite{MF2} using the theory of Iterated Function System (IFS). FIFs  were defined by Celik et al. on $SG$ in \cite{Celik} and by Ruan \cite{Ruan3} on the basis of post-critically finite (p.c.f.) self-similar sets, presented by Kigami \cite{Kig} to investigate fractal analysis. On $SG$, Ri and Ruan \cite{Ruan4} have determined the necessary and sufficient conditions for the uniform FIFs to have finite energy. Very recently, Ri \cite{Ri2} generated the graphs of FIFs on $SG$. The comprehensive framework of fractals and associated geometry, as well as various concepts of dimension and techniques for computing them, were discussed by Falconer \cite{Fal} and the geometrical aspects of fractals were examined. Navascu\'es  \cite{M2} has presented a technique for defining non-smooth versions of classical approximants using FIFs. Recently, Priyadarshi \cite{Pr1} has presented a technique to find the lower bounds for the Hausdorff dimension of the set of every infinite complex continued fractions with partial numerators is one, whereas Jha and Verma \cite{Verma21}  have calculated the exact dimension of the graph of the $\alpha$-fractal functions.

 On the $SG$, Sahu and Priyadarshi \cite{SP} derived certain constraints for box dimensions of harmonic functions. In \cite{vermabv}, Verma and Sahu have evaluated the dimensions of some functions through  oscillation spaces and also developed certain notions of bounded variation on the $SG$ and based on this. They also derived various dimensional consequences generalizing Liang's result \cite{Liang1} to fractal domains. In \cite{AGN}, Agrawal et al. have developed the concept of dimension preserving approximation for real-valued bivariate continuous functions, defined on a rectangular domain. Moreover, Agrawal and Som \cite{EPJS} have studied the FIFs and its box dimension corresponding to a continuous function defined on the $SG$ and have discussed the $\mathcal{L}^p$-approximation related results on the $SG$ in \cite{LP}. Chandra and Abbas \cite{SS2} investigated bivariate FIFs and some of their significant properties. Furthermore, for several continuous functions on a rectangular domain, Chandra and Abbas \cite{Chana} have presented a comprehensive analysis of the fractal dimension of the graph of the mixed Riemann-Liouville fractional integral. Numerical examples are added by Gowrisankar and Uthayakumar \cite{Gme} to examine the Riemann–Liouville fractional calculus of FIFs. 
 \par
Laplacian on $SG$ can be defined in two ways. The first approach is based on the probability theory and the second approach is based on the calculus, both are given by Kigami\cite{Kig}. However, we follow the second approach and  develop the approximation theory on the well-known fractal $SG$  in this paper.
\par 
The following is the outline of the paper. Section 1 contains a brief introduction. Section 2 deals with the preliminaries that are necessary for the paper. Section 3 explores some significant properties associated with the fractal operator on the space of all continuous functions defined on $SG$, denoted by $\mathcal{C}(SG)$. Section 4 provides some results related to constrained approximation with fractal polynomials. In Section 5, we study the best approximation properties of fractal polynomials defined on $SG$. Section 6 provides some remarks on the class of polynomials defined on $SG$. In Section 7 provides the bounds for the dimension of the graph of of $\alpha$- fractal function defined on $SG$.
\par 

\section{Preliminaries}
 
 \subsection{Attractors and measures}
 
 One of the most fundamental concepts that emerge from IFS is the concept of an invariant set. IFS is composed of a finite family of contraction maps. \par
Let $\mathcal{K}(X)$ be a collection of all non-empty compact subsets of $X$, where $(X,d) $ is a metric space.
 For all $k=1, \ldots, K$, let $\mathcal{S}_{k} : X \rightarrow X$ be a finite family of contraction maps with a contraction ratio $\alpha_k$ respectively. The family $\left(\mathcal{S}_{k}\right)_{k=1}^{K}$ is called an IFS. The function $\mathcal{S}$ : $\mathcal{K}(X) \rightarrow \mathcal{K}(X)$ defined by $\mathcal{S}(B)=\bigcup_{k=1}^{K} \mathcal{S}_{k}(B)$ is called the associated Hutchinson operator. The metric space $(\mathcal{K}(X), h)$ is complete with respect to the Hausdorff metric $h$ if and only if $(X, d)$ is complete for the contraction ratio $\alpha$ of $\mathcal{S}$ is $\max\{\alpha_k\}_{k=1}^{K}$. A set $A \in \mathcal{K}(X)$ is said to be an attractor of the IFS whenever $\mathcal{S}(A)=A$.\\
  For more details, the reader may refer to \cite{MF2}. Let us denote the Hausdorff dimension, lower box dimension and upper box dimension of $A$ by $\operatorname{dim}_{H}(A)$,  $\underline{\operatorname{dim}}_{B}(A)$ and $\overline{\operatorname{dim}}_{B}(A)$, respectively and  the graph of $f$ by $Gr(f)$, i.e.,  $Gr(f)=$ $\{(t, f(t)): t \in \mathrm{SG}\}$.  
\\

\begin{definition}[\cite{MF2}, Definition 5.1]
Let  $(X, d)$ be a compact metric space and $\mu_{p}$ be a Borel Measure on  $X$. If $\mu_{p}(X)=1$, then $\mu_{p}$ is called the Borel Probability Measure.
\end{definition} 
We denote self-similar measure by $\mu_p$ and this measure arises from the IFS with probability vectors, which is the fixed point of the Markov operator, i.e., the invariant measure of the IFS with probability vectors.
 To understand the self-similar measure, the readers may refer to \cite{MF2}.\\

     \subsection{Fractal interpolation function on SG}
     
     Let $A$ be the notation of an equilateral triangle. Let $V_{0}=\{p_1,p_2,p_3\}$ be the vertices of $A$. For the vertices $V_{0}$, the contraction mappings $L_{i}: \mathbb{R}^2 \to \mathbb{R}^2$ are defined by 
$$L_{i}(t)=\frac{1}{2}(p_{i}+t),$$
where $p_{i} \in V_{0}$ and $\{\mathbb{R}^2: L_1, L_2,L_3\}$ be the IFS and produces $SG$ as an attractor, that is, $ SG=\bigcup_{i=1}^{3} L_i(SG) .$
For $N \in \mathbb{N}$, $I^N$ represents the set of all the words of length $N$, where $I=\{1,2,3\}$, that is, if $\boldsymbol{\omega} \in I^N$, then $\boldsymbol{\omega}=\omega_1, \dots ,\omega_N$, where $\omega_i \in \{1,2,3\}.$ For $\boldsymbol{\omega} \in I^N$, $L_{\boldsymbol{\omega}}$ is defined by $$L_{\boldsymbol{\omega}}= L_{\omega_1} \circ \dots \circ L_{\omega_N}.$$ 

     In \cite{Hut}, it is noted that there exists a unique Borel probability measure $\mu_{p}$ supported on SG associated with the IFS  $\{\mathbb{R}^2,L_1,L_2,L_3\}$ and probability vector $p=\left(\frac{1}{3},\frac{1}{3},\frac{1}{3}\right)$ satisfying: \begin{equation*} \begin{aligned}  \mu_{p} &= \frac{1}{3} \sum_{i=1}^{3} \mu_{p} \circ L_{i}^{-1}. \end{aligned} \end{equation*}
       
     Set $V_N=\{p_1,p_2,p_3, L_{\boldsymbol{\omega}}(p_2),L_{\boldsymbol{\omega}}(p_3),L_{\boldsymbol{\omega}}(p_1):\boldsymbol{\omega} \in I^N\}.$ Let $N\in \mathbb{N}$ and $|{\boldsymbol{\omega}}|=N$ and if we take the union of images of $V_{0}$ for the iterations $L_{\boldsymbol{\omega}}$, then we get the set of N-th stage vertices $V_N$ of $SG$. Now define $V_{*}= \bigcup_{N=1} ^{\infty}V_N.$
     Let $f:SG \rightarrow \mathbb{R}$ be a function on SG
     and $K=SG \times \mathbb{R}$ . Define maps $W_{\boldsymbol{\omega}} :K \rightarrow K$ by $$ W_{\boldsymbol{\omega}}(t,x)=\Big(L_{\boldsymbol{\omega}}(t),F_{\boldsymbol{\omega}}(t,x)\Big),~~ {\boldsymbol{\omega}} \in \{1,2,3\}^N ,$$
     where $F_{\boldsymbol{\omega}}(t,x): SG \times \mathbb{R} \rightarrow \mathbb{R}$ must satisfy the following condition: 
      $$ | F_{\boldsymbol{\omega}}(.,x)-F_{\boldsymbol{\omega}}(.,x')| \le c |x - x'|,$$ with $c<1$ and
     $ F_{\boldsymbol{\omega}}(p_j,f(p_j))=f(L_{\boldsymbol{\omega}}(p_j)).$
     In particular,  $$ F_{\boldsymbol{\omega}}(t,x)= \alpha_{\boldsymbol{\omega}}(t) x+f(L_{\boldsymbol{\omega}}(t))-\alpha_{\boldsymbol{\omega}}(t)b(t),$$ where $b:SG \to \mathbb{R}$ is a map satisfying $b(p_j)=f(p_j),~j= 1,2,3$ and for  all $\boldsymbol{\omega} \in I^N$, $\alpha_{\boldsymbol{\omega}}: SG \to \mathbb{R}$ is a continuous function with $\|\alpha_{\boldsymbol{\omega}}\|_{\infty} < 1.$ 
     Finally, the system $\{K;W_{\boldsymbol{\omega}}, \boldsymbol{\omega} \in I^N\}$ is an IFS. From \cite{EPJS}, we get a unique self-referential function $f^{\alpha}$ which satisfies
   \begin{equation}\label{Fnleq1}
     f^{\alpha}(t)= f(t)+\alpha_{\boldsymbol{\omega}}(L_{\boldsymbol{\omega}}^{-1}(t))~ f^{\alpha}\big(L_{\boldsymbol{\omega}}^{-1}(t)\big)- \alpha_{\boldsymbol{\omega}}(L_{\boldsymbol{\omega}}^{-1}(t))~b\big(L_{\boldsymbol{\omega}}^{-1}(t)\big),
  \end{equation}
  for all $t \in L_{\boldsymbol{\omega}}(SG),~  \boldsymbol{\omega}~~ \in I^N,$ and the  $Gr(f^{\alpha})$ is the attractor of IFS $\{K;W_{\boldsymbol{\omega}}, \boldsymbol{\omega} \in I^N\}.$ This type of fractal functions are widely known as $\alpha$-fractal functions \cite{M2, Verma21}.\\
 
\subsection{Energy and Laplacian on SG}
 For the analysis on fractals, we need the following basic concepts and properties in this section. We refer to \cite{stri} for further information.\\
\subsection{Energy}  
We begin with the construction of  a complete graph ${\Gamma_0}$ via vertex set $V_0$. The following procedure is used to recursively construct the graph ${\Gamma_m}$. First we construct the graph $\Gamma_{m-1}$ using the vertex set $V_{m-1}$  for some $m\ge 1$. The graph ${\Gamma_m}$ on $V_m$ is defined as follows: The edge relation $t \sim_m z$ exists for any  $t,z\in V_m$
 if and only if $t = L_{i}(t^{'}),~ z = L_{i}(z^{'})$ with
$t^{'}\sim_{m-1}z^{'}$ and $i\in I$. Equivalently, $ t \sim_{m} z$ if and only if there exists $\boldsymbol{\omega} \in I^N$ such that $t,z \in L_{\boldsymbol{\omega}}(V_{0})$.
\begin{definition}
Let $m \in \mathbb{N} \cup \{0\}$, the graph energy at the $m^{th}$ level on  $\Gamma_m$ is denoted by $E_{m}$ and defined as follows:

$$E_m(f) =\bigg(\frac{5}{3}\bigg)^{m}\sum_{t\sim_{m}z}\big(f{(t)}-f(z)\big)^2,$$
which satisfy $E_{m-1}(f)= \min E_{m}(\tilde{f})$, where the minimum is computed through every $\tilde{f}$ such that $\tilde{f}|_{V_{m-1}}=f$ for each $f:V^{*}\to \mathbb{R}$ with
$m\geq 1$. It is worth noting that  $\big(E_{m}(f)\big)_{m=0}^{\infty}$ is increasing  sequence for all $f$.

The limit
$$E(f) := \lim_{m\to\infty}E_m(f)$$
is referred to as the energy of $f$ on $V_{*}$. Note that if $f$ has finite energy, then
$E(f)< \infty.$\\
Note that every $f$ on $V_{*}$ whose energy is finite are uniformly continuous. Since $V_{*}$ is dense in $SG$, which immediately implies that $f$ has a unique extension as a continuous function on $SG$. We shall denote $f \in \text{dom}(\mathcal{E})$ as a continuous function $f$ on $SG$ with $E(f)< +\infty$. We now define for each $f, g \in \text{dom}(\mathcal{E})$
$$\mathcal{E}(f,g)=  \lim_{m\to \infty}\bigg(\frac{5}{3}\bigg)^{m}\sum_{x\sim_{m}y}(f(x)-f(y))(g(x)-g(y)).  $$
\end{definition}
\begin{definition}
Let $f\in \mathcal{C}(SG)$ such that $E_{m-1}(f)=E_{m}(f)$ for all $m\geq 1$, then $f$ is
a harmonic function on $SG$.
\end{definition}
\begin{definition}
Let $f \in \text{dom}(\mathcal{E})$ and $g$ be continuous. Then $f \in \text{dom}\Delta$ with $\Delta f=g$ if
$$\mathcal{E}(f, h) = -\int_{SG} gh ~d\mu_{p} ~~\text{for all}~~ h\in \text{dom}_{0}(\mathcal{E}),$$
where $\text{dom}_{0}(\mathcal{E})=\{f \in \text{dom}(\mathcal{E}): f|_{V_{0}}=0\}.$\\
The Laplacian of $f$ can also be determined using a pointwise formula.
  Define a graph Laplacian $\Delta_m$ on $\Gamma_m$ by
$$ \Delta_{m}f(x):= \sum_{y\sim_{m}x}\big(f(y)-f(x)\big), ~x \in V_{m} \backslash V_{0}.$$
\end{definition} 

Using Theorems $(2.2.1)$ and $(2.2.12)$ in \cite{stri}, we get the following lemma.
\begin{lemma}
  Let $f \in \text{dom}(\Delta)$, then the pointwise formula
    \begin{equation}\label{pointwise}
        \Delta f(x)=\frac{3}{2} \lim_{m \to \infty} 5^{m}\Delta_{m}f(x)
    \end{equation}
holds with the limit uniform across $V_{*} \backslash V_{0}$. Conversely, suppose $f$ is a continuous function and the right side of equation (\ref{pointwise}) converges uniformly to a continuous function on  $V_{*} \backslash V_{0}$. Then 
$f \in \text{dom}(\Delta)$ and equation (\ref{pointwise}) holds.
\end{lemma}
It can be seen that the Laplacian is a linear operator from the above lemma.
Let $h$ be a harmonic function. Now, using the well-known rule $``\frac{1}{5}-\frac{2}{5}"$, we have $\Delta_{m}h(x)=0$ for any $m\in \mathbb{
N}$ and $x\in V_{m} \backslash V_{0}$ so that $\Delta h =0$.\\
\subsection{Polynomials on SG} The space of polynomials on the unit interval may be represented as the space of solutions to $\Delta^{k}=0$ for some $k$. So one can define a polynomial on $SG$  with standard Laplacian discussed in \cite{stri} as the solution of the same equation. We define
$$\mathcal{H}_{k}(SG)=\{f : \Delta^{k+1}f=0\}$$. These functions are referred as a multi-harmonic. In the case of  $\mathcal{H}_{0}$, the class of harmonic functions is three-dimensional and a function in $\mathcal{H}_{0}$ is obtained uniquely by boundary values. For $\mathcal{H}_{1}$, the class of bi-harmonic functions is six- dimensional. For $\mathcal{H}_{2}$, the class of tri-harmonic functions is nine-dimensional. For $\mathcal{H}_{k}$, the class of multi-harmonic functions is $3(k+1)$ dimensional.\\
\par The Weierstrass Approximation Theorem shows that the continuous real-valued functions on a compact interval can be uniformly approximated by polynomials. 
Whereas the polynomials defined on $SG$, which are obtained from the class of multi-harmonic functions, do not follow the Weierstrass Approximation Theorem.\\

\begin{definition}\label{lowerHolder}
Let $f: SG \rightarrow \mathbb{R}$ be a function defined on SG. Then $f$ is called lower H\"older continuous function if there exists $\delta_0 >0$ such that for all $x\in SG$ and $\delta < \delta_0$, there exists $y\in SG$ with $\lVert x-y \rVert_2\leq \delta$ such that 
\[\lvert f(x)-f(x) \rvert \geq K \lVert x-y \rVert_{2}^{\sigma},\]
where $K\geq 0$ is a constant and $0<\sigma \leq 1$ is H\"older exponent.
\end{definition}


\begin{theorem}
Let  $f: SG \rightarrow \mathbb{R}$ be a lower H\"older continuous function. Then the energy of $f$ is infinite.
\end{theorem}

\begin{proof}
Since $f$ is a lower H\"older continuous function, then using Definition \ref{lowerHolder}, for every $x\in SG$ there exists $y\in SG$ such that $\lVert x-y \rVert_{2}\leq 1$. Since $0< \sigma \leq 1$. Then
\begin{equation}\label{eq3}
    \lVert x-y \rVert_{2}^{\sigma}\geq \lvert x-y \rVert_2.
\end{equation}
For $n\in \mathbb{N}$ and for all $x, y \in V_{n}$, the $n^{th}$ level energy of $f$, $E_{n}(f)$ is 
$$
E_{n}(f)=\left(\frac{5}{3}\right)^{n} \sum_{x \sim_{n} y}\|f(x)-f(y)\|^{2}
.$$
 Using Definition \ref{lowerHolder}, we get
\begin{equation}
    E_{n}(f) \geq\left(\frac{5}{3}\right)^{n} K \sum_{x \sim_{n} y}\|x-y\|_{2}^{2\sigma}.
\end{equation}
Now, from \eqref{eq3}, we have 
\begin{equation}\label{en}
    E_{n}(f) \geq\left(\frac{5}{3}\right)^{n} K \sum_{x \sim_{n} y}\|x-y\|_{2}^{2}.
\end{equation}

For $x, y \in V_{0}$, we get
$$\sum_{x \sim_{0} y}\|x-y\|_{2}^{2}=1+2\left(\frac{1}{2}\right)^{2}.$$
At $n^{th}$ level of $SG$, we have

\begin{equation}
     \begin{aligned} 
     \sum_{x \sim_{n} y}\|x-y\|_{2}^{2} &=3^{n}\left[\left(\frac{1}{2^{n}}\right)^{2}+2\left(\frac{1}{2^{n+1}}\right)^{2}\right] \\ &=\frac{3^{n+1}}{2^{2 n+1}} .
     \end{aligned}
\end{equation}
Putting the above value in \eqref{en}, we get 
$$
E_{n}(f) \geq 3K \frac{5^{n}}{2^{2 n+1}}
$$
Finally,
$$
\lim _{n \rightarrow \infty} E_{n}(f)=\infty \text {. }
$$
Therefore, energy of the lower H\"older continuous function $f$ defined on $SG$ is infinite, that is,
$$
E(f)=\infty \text {. }
$$
This completes the proof.
\end{proof}


\section{Associated Fractal Operator on $\mathcal{C}(SG)$}
Let $f \in \mathcal{C}(SG)$ be arbitrarily chosen. If we choose $b= Lf,$ 
   where $L:\mathcal{C}(SG) \rightarrow \mathcal{C}(SG)$ is a bounded linear operator such that  $(Lf)|_{V_0}=f|_{V_0},$ in the construction of $\alpha$-fractal function $f^\alpha_{n,L}$.
    Then, the following functional equation is satisfied by the corresponding fractal function $f^\alpha_{n,L}$:
    $$ f^{\alpha}_{n,L}(t)= f(t)+\alpha_{\boldsymbol{\omega}}(L_{\boldsymbol{\omega}}^{-1}(t)) (f^{\alpha}_{n,L}- Lf)\big(L_{\boldsymbol{\omega}}^{-1}(t)\big), ~ ~ ~~\forall~~ t \in L_{\boldsymbol{\omega}}(SG),~~ \boldsymbol{\omega} \in I^{N}.$$
    \begin{definition}
   
An $\alpha$-fractal operator $\mathcal{F}^{\alpha}_{n,L}: \mathcal{C}(SG) \rightarrow \mathcal{C}(SG)$ is defined by $$\mathcal{F}^{\alpha}_{n,L}(f)= f_{n,L}^{\alpha}.$$ 
   For the simplicity, we use $f^\alpha$  instead of
$f_{n,L}^{\alpha}$ and  $\mathcal{F}^{\alpha}$ instead of ${F}^{\alpha}_{n,L}$.
    
    \end{definition}
    \begin{remark}\label{remnew1}
   
   We can see that $f_{n,L}^{\alpha}$ and the operator $\mathcal{F}_{n,L}^{\alpha}$ have the following interpolatory property:
   
   $$ \mathcal{F}_{n,L}^{\alpha}(f)(x)= f(x),~~ \forall~x \in V_N.$$ 
    \end{remark}

The continuous dependency of parameters, i.e., original function $f(x,y)$, base function $b(x,y)$ and scaling factor $\alpha(x,y)$  on $\alpha$- fractal functions $f^{\alpha}(x,y)$ can be directly seen through the graphs, which are drawn below.
 
       \begin{figure}[h!]
       \begin{minipage}{0.5\textwidth}
       \includegraphics[width=1.0\linewidth]{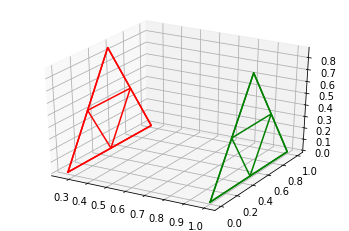}
       {\hspace*{1cm}$Gr(f^{\alpha})$ at $1^{st}$ iteration}
       \end{minipage}\hspace*{0.3cm}
       \begin{minipage}{0.5\textwidth}
       \includegraphics[width=1.0\linewidth]{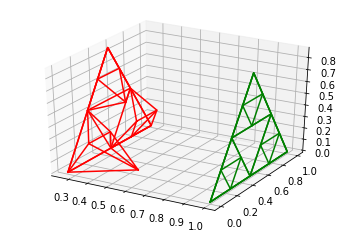}
       {\hspace*{3cm}$Gr(f^{\alpha})$ at $2^{nd}$ iteration}
       \end{minipage}
       \begin{minipage}{0.5\textwidth}
       \includegraphics[width=1.0\linewidth]{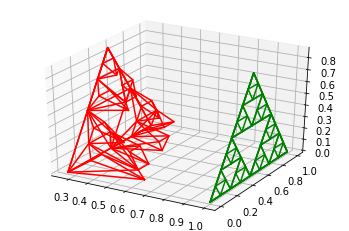}
       {\hspace*{3cm}$Gr(f^{\alpha})$ at $3^{rd}$ iteration}
       \end{minipage}\hspace*{0.3cm}
       \begin{minipage}{0.5\textwidth}
       \includegraphics[width=1.0\linewidth]{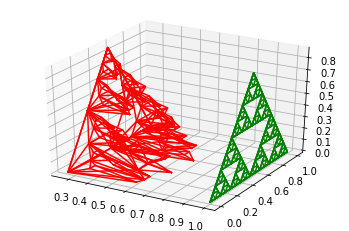}
       {\hspace*{3cm}$Gr(f^{\alpha})$ at $4^{th}$ iteration}
       \end{minipage}
       \caption{$b(x,y)= (xy+113)/432-x(x-y+1.22)(x-1)y(y-3^{1/2}/2)$,     $f(x,y)=(xy+113)/432$  and $\alpha=0.7$.}
       \label{fig1}
       \end{figure}
\begin{figure}[h!]
\begin{minipage}{0.5\textwidth}
                       \includegraphics[width=1.0\linewidth]{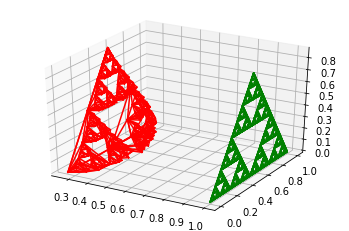}
                       {\hspace*{1cm}$Gr(f^{\alpha})$ at $\alpha = 0.3$}
                       \end{minipage}\hspace*{0.3cm}
                       \begin{minipage}{0.5\textwidth}
                       \includegraphics[width=1.0\linewidth]{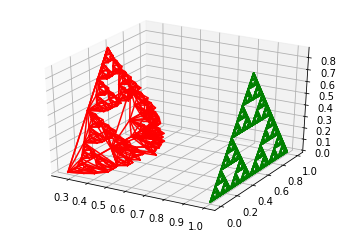}
                       {\hspace*{3cm}$Gr(f^{\alpha})$ at $\alpha = 0.4$}
                       \end{minipage}
                       \begin{minipage}{0.5\textwidth}
                       \includegraphics[width=1.0\linewidth]{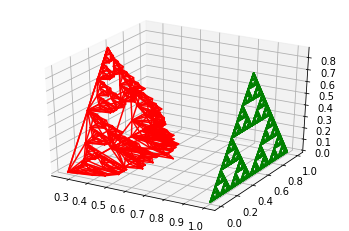}
                       {\hspace*{3cm}$Gr(f^{\alpha})$ at $\alpha = 0.5$}
                       \end{minipage}\hspace*{0.3cm}
                       \begin{minipage}{0.5\textwidth}
                       \includegraphics[width=1.0\linewidth]{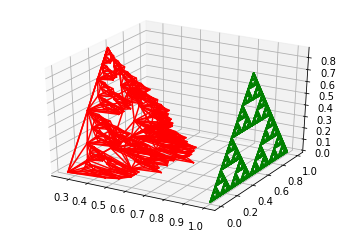}
                       {\hspace*{3cm}$Gr(f^{\alpha})$ at $\alpha = 0.6$}
                       \end{minipage}
                       \caption{$Gr(f^{\alpha})$ for the various values of the  $\alpha$, where $b(x,y)= (xy+113)/432-x(x-y+1.22)(x-1)y(y-3^{1/2}/2)$,     $f(x,y)=(xy+113)/432$.}
                       \label{fig2}
                       \end{figure}

\begin{lemma}[Lemma 1, \cite{CC}]
	Let $(Y,\|.\|)$ be a Banach space and $B: Y \to Y$ be a linear operator. Suppose $\lambda_1, \lambda_2 \in [0,1)$ such that
	$$ \|By-y\| \le \lambda_1 \|y\| + \lambda_2 \|By\|, \quad \forall~~ y \in Y.$$
Then $B$ is a topological isomorphism, and
	$$\frac{1-\lambda_2}{1+\lambda_1} \|y\| \le  \| B^{-1}y\| \le \frac{1+\lambda_2}{1-\lambda_1}  \|y\|,\quad \forall~~y \in Y.$$
\end{lemma}
\begin{theorem}\label{thmtopiso}
Let $L$ be an operator as mentioned above such that $\|\alpha\|_\infty < \|L\|^{-1}$. Then  $\mathcal{F}^\alpha:
	\mathcal{C}(SG) \to \mathcal{C}(SG)$ is a topological isomorphism.
\end{theorem}
\begin{proof}
Using Equation (\ref{Fnleq1}), we have
	\begin{equation*}
	\begin{split}
\big	\| f - \mathcal{F}^\alpha(f)\big\|_\infty \le&~ \|\alpha\|_\infty  \big\|\mathcal{F}^\alpha(f) -  Lf  \big\|_\infty \\ \le &~  \|\alpha\|_\infty \Big[\big\|\mathcal{F}^\alpha(f)\big\|_\infty+\| L f\|_\infty \Big]\\ = &~\|\alpha\|_\infty \Big[\big\|\mathcal{F}^\alpha(f)\big\|_\infty+\| L\| \|f\|_\infty \Big].
	\end{split}
	\end{equation*}
Note that $\|\alpha\|_\infty < 1$ and $\|\alpha\|_\infty < \|L\|^{-1}$. Hence, $\mathcal{F}^\alpha$ is a topological isomorphism by the preceding lemma.
\end{proof}
\begin{definition}
The space $\mathcal{S}(H_0,V_m,\mathbb{R})$ of real-valued piecewise harmonic functions of level $m$ is defined to be the space of continuous functions $f$ such that $f\circ L_{\boldsymbol{\omega}}$ is harmonic for all $\boldsymbol{\omega} \in I^{N},$ where $L_{\boldsymbol{\omega}}(t)= \frac{1}{2^m}t + \sum_{k=1}^{m} \frac{1}{2^k} p_{{\omega}_k}.$
\end{definition}

The upcoming theorem concerning the density of real-valued piecewise harmonic functions is simple to prove, by \textup{\cite[Theorem 1.4.4]{stri}}.
\begin{theorem}\label{Denthm1}

Let $f\in \mathcal{C}(SG)$, then there exists a sequence $(h_m) \in \mathcal{S}(H_0,V_m,\mathbb{R})$ satisfying $h_m |_{V_m} = f|_{V_m}$ such that $h_m \to f$  uniformly.


\end{theorem}

\begin{definition}
Corresponding to the aforementioned class $\mathcal{S}(H_0,V_m,\mathbb{R})$, we define the bounded linear operator $\mathcal{F}^\alpha:\mathcal{C}(SG) \to \mathcal{C}(SG)$. 
$\mathcal{F}^\alpha \big( \mathcal{S}(H_0,V_m,\mathbb{R})  \big)$ represents the class of all (real-valued) fractal piecewise harmonic functions, i.e., a continuous function $h^\alpha: SG \to \mathbb{R}$ is a fractal piecewise harmonic function if 
$h^\alpha= \mathcal{F}^\alpha (h)$ for  $h \in \mathcal{S}(H_0,V_m,\mathbb{R})$. 
\end{definition}
\begin{proposition}\label{ee}
 We  have the following error estimate:
$$\|f^{\alpha}-f\|_{\infty}\leq \frac{\|\alpha\|_{\infty}}{1-\|\alpha\|_{\infty}}\|f-b\|_{\infty
}.$$
\end{proposition}
\begin{proof}
The proof follows from Equation \ref{Fnleq1}.
\end{proof}

\section{Constrained approximation}
Let us now present the following notations:

\begin{align}
&m_{*}=\min \{b(t): t \in SG\}, \quad m_{\boldsymbol{\omega}}=\min \left\{f(t): t \in L_{\boldsymbol{\omega}}(SG),~~  \boldsymbol{\omega} \in I^N \right\}, \\
& M^{*}=\max \{b(t): t \in SG\}, \quad M_{\boldsymbol{\omega}}=\max \left\{f(t): t \in L_{\boldsymbol{\omega}}(SG),~~  \boldsymbol{\omega} \in I^N\right\}.
\end{align}

\begin{theorem}\label{minmax}
Let $f \in \mathcal{C}(SG)$ be such that $0\leq f(t)\leq \tilde{M}$ for all $t\in SG$. If $\alpha_{\boldsymbol{\omega}}\in \mathcal{C}(SG)$ is so chosen that $\|\alpha_{\boldsymbol{\omega}}\|_{\infty}<1$ and for all $\boldsymbol{\omega} \in I^{N}$ and for all $t\in SG$,
$$ \max \bigg\{\frac{-m_{\boldsymbol{\omega}}}{\tilde{M}-m_{*}}, \frac{M_{\boldsymbol{\omega}}-\tilde{M}}{M^*}\bigg\}\leq \alpha_{\boldsymbol{\omega}}(t) \leq \min \bigg\{\frac{m_{\boldsymbol{\omega}}}{M^{*}}, \frac{\tilde{M}-M_{\boldsymbol{\omega}}}{\tilde{M}-m_{*}}\bigg\},$$
then the corresponding $f^{\alpha}\in \mathcal{C}(SG)$ also satisfies $0\leq f^{\alpha}(t)\leq \Tilde{M}~\forall~  t\in SG$. 
\end{theorem}
\begin{proof}
Note that the functional equation is satisfied by $f^{\alpha}$.
\begin{equation}\label{vis}
     f^{\alpha}(t)= F_{\boldsymbol{\omega}}(L_{\boldsymbol{\omega}}^{-1}(t),f^{\alpha}(\L_{\boldsymbol{\omega}}^{-1}(t)) ~~~\forall t \in L_{\boldsymbol{\omega}}(SG),~~  \boldsymbol{\omega} \in I^N,
\end{equation}

 where $ F_{\boldsymbol{\omega}}(t,x)= \alpha_{\boldsymbol{\omega}}(t) x+f(L_w(t))-\alpha_{\boldsymbol{\omega}}(t)b(t)$ and it satisfies the interpolatory condition   $f^{\alpha}|_{V_N}=f|_{V_N}$. We need to show $0\leq f^{\alpha}(t) \leq \tilde{M}$ for all $t \in SG$. Since $SG= \bigcup_{{\boldsymbol{\omega}} \in \{1,2,3\}^{N} }L_{\boldsymbol{\omega}}(SG)$, a piecewise harmonic function $f^{\alpha}$ is constructed iteratively through the functional equation is given by (\ref{vis}) and $f^{\alpha}|_{V_N}=f|_{V_N}$, this is sufficient to prove that  $\forall$ $\boldsymbol{\omega}$, and $(t,x) \in SG \times [0, \tilde{M}]$ 
 $$0 \leq F_{\boldsymbol{\omega}}(t,x)\leq \tilde{M}.$$
 Let $(t,x) \in  SG \times [0, \tilde{M}]$ and
   $$q_{\boldsymbol{\omega}}(t)= f(L_{\boldsymbol{\omega}}(t))-\alpha_{\boldsymbol{\omega}}(t)b(t).$$
   First, we consider $0\leq \alpha_{\boldsymbol{\omega}}(t)<1$. This assumption with $0\leq x \leq \tilde{M}$ yields,
   \begin{equation}
       \begin{aligned}
         q_{\boldsymbol{\omega}}(t) \leq F_{\boldsymbol{\omega}}(t,x)&=\alpha_{\boldsymbol{\omega}}(t)x+q_{\boldsymbol{\omega}}(t)\\
           & \le \alpha_{\boldsymbol{\omega}}(t)\tilde{M} + q_{\boldsymbol{\omega}}(t).
             \end{aligned}
   \end{equation} 
   Therefore $0 \le F_{\boldsymbol{\omega}}(t, x)\leq \tilde{M}$ holds if
   \begin{equation}\label{EQ1}
         0 \leq q_{\boldsymbol{\omega}}(t)\leq \tilde{M}(1-\alpha_{\boldsymbol{\omega}}(t)).
   \end{equation}
 
   Keeping in mind the fact that $f(L_{\boldsymbol{\omega}}(t)) \geq m_{\boldsymbol{\omega}}$ and $b(t)\leq M^{*}$ for all $t\in SG$, it can be readily verified that $\alpha_{\boldsymbol{\omega}}(t) \leq \frac{m_{\boldsymbol{\omega}}}{M^{*}}$ ensures the first
inequality in (\ref{EQ1}). Recall that if $M^{*}$ is zero, then no further conditions on $ \alpha_{\boldsymbol{\omega}}(t)$ are required to ensure
$ q_{\boldsymbol{\omega}}(t) \geq 0$. Similarly, using $f(L_{\boldsymbol{\omega}}(t))\leq M_{\boldsymbol{\omega}}$ and $b(t)\geq m^{*}~  \forall~ t \in SG$, we assert that $ \alpha_{\boldsymbol{\omega}}(t) \leq \frac{\tilde{M}-M_{\boldsymbol{\omega}}}{\tilde{M}-m^{*}}$ guarantees the second inequality in (\ref{EQ1}). We thus choose $\alpha_{\boldsymbol{\omega}}$ as shown below to ensure that (\ref{EQ1}) holds,

   \begin{equation}
        \alpha_{\boldsymbol{\omega}}(t) \leq \min \bigg\{\frac{m_{\boldsymbol{\omega}}}{M^{*}}, \frac{\tilde{M}-M_{\boldsymbol{\omega}}}{\tilde{M}-m_{*}}\bigg\}.
\end{equation}
  
  Now we assume $-1<\alpha_{\boldsymbol{\omega}}(t)<0$. Let $0\leq z \leq \tilde{M}$, then by using simple steps one gets 
  \begin{equation}
     \alpha_{\boldsymbol{\omega}}(t) \tilde{M} + q_{\boldsymbol{\omega}}(t)\leq \alpha_{\boldsymbol{\omega}}(t)x + q_{\boldsymbol{\omega}}(t) \leq q_{\boldsymbol{\omega}}(t),
 \end{equation}
 consequently, for $0\leq F_{\boldsymbol{\omega}}(t,x) \leq \tilde{M}$, this is sufficient to show that 
 \begin{equation}\label{EQ2}
 -\alpha_{\boldsymbol{\omega}}(t) M \leq  q_{\boldsymbol{\omega}}(t) \leq \tilde{M}.
 \end{equation}
 The same calculations as for (\ref{EQ1}) indicate that (\ref{EQ2}) is satisfied if
 
$$\alpha_{\boldsymbol{\omega}}(t)\geq \max \bigg\{\frac{-m_{\boldsymbol{\omega}}}{\tilde{M}-m_{*}}, \frac{M_{\boldsymbol{\omega}}-\tilde{M}}{M^*}\bigg\}.$$
Hence, we have established the assertion.
   \end{proof}
 
\begin{remark}
If $f \in \mathcal{C}(SG)$ with $f(t)\leq0$ for all $t\in SG$, then  its fractal counterpart $f^{\alpha}$ can be constructed and for all $t \in SG$ that holds $f^{\alpha}(t)\leq 0$. To obtain this, we use the previous theorem for the positive function $\hat{f}=-f$ and the associated function $\hat{s}=-s$. We take the scaling function $\alpha_{\boldsymbol{\omega}} \in SG$ satisfying  $\|\alpha_{\boldsymbol{\omega}}\|_{\infty} < 1$ and

$$ \max \bigg\{\frac{-M_{\boldsymbol{\omega}}}{\tilde{m}-M^{*}}, \frac{m_{\boldsymbol{\omega}}-\tilde{m}}{m_*}\bigg\}\leq \alpha_{\boldsymbol{\omega}}(t) \leq \min \bigg\{\frac{M_{\boldsymbol{\omega}}}{m_{*}}, \frac{\tilde{m}-m_{\boldsymbol{\omega}}}{\tilde{m}-M^{*}}\bigg\},$$ 
ensures $m\leq f^{\alpha}(t)\leq 0$ for all $t \in SG$.
\end{remark}
\begin{theorem}\label{cc}
Let $f\in \mathcal{C}(SG)$ and $f^{\alpha}$ be an $ \alpha$- fractal function corresponding to $f$. Then $f^{\alpha}(t)\leq f(t)$ for all $t \in SG$ provided $\alpha(t)\geq 0$ and $b(t)\geq f(t)$ for every $t \in SG$.
\end{theorem} 
\begin{proof}
Note that $f$ is a fixed point of the RB-operator and satisfies the
following self-referential equation
$$f^{\alpha}(L_{\boldsymbol{\omega}}(t))=f(L_{\boldsymbol{\omega}}(t))+\alpha_{\boldsymbol{\omega}}(t)(f^{\alpha}-b)(t),$$
for every $t\in SG$ and $\boldsymbol{\omega} \in I^{N}$. The previous theorem showed that in establishing  $(f^{\alpha}-f)(t)\leq 0$ for all $t\in SG$, it is sufficient to verify
that $(f^{\alpha}-f)(t)\leq 0$ is satisfied at the points on $SG$ obtained at the $(i+1)$-th iteration whenever $(f^{\alpha}-f)\leq 0$ holds for the points on $SG$ at the $i$-th iteration. We can find that the above condition is equivalent to $(f^{\alpha}-f)(L_{\boldsymbol{\omega}}(t))\leq 0$ for every $\boldsymbol{\omega} \in I^{N}$ whenever $(f^{\alpha}-f)(t)\leq 0$. Therefore, using the above, we have 
\begin{equation}
\begin{split}
    (f^{\alpha}-f)(L_{\boldsymbol{\omega}}(t))\leq 0= &~ \alpha_{\boldsymbol{\omega}}(t)(f^{\alpha}-b)(t)\\
= &~ \alpha_{\boldsymbol{\omega}}(t)(f^{\alpha}-f)(t) + \alpha_{\boldsymbol{\omega}}(t)(f-b)(t).
\end{split}
\end{equation}

If we set up $\alpha$ and $b$ as above mentioned theorem then we can determine $(f^{\alpha}-f)(L_{\boldsymbol{\omega}}(t))\leq 0$ by using the assumption $(f^{\alpha}-f)(t)\leq 0$, which completes the proof.

\end{proof}

\begin{remark}

Using the same steps, one can verify that if $\alpha(t)\geq 0$ and $b(t)\leq f(t)$, then  $f^{\alpha}(t)\geq f(t)$  for every $t\in SG$.
\end{remark}

\begin{theorem}
Let $f \in \mathcal{C}(SG)$ satisfying
$f(t)\geq 0$ all $t\in SG$. Let $\epsilon>0$, then we can choose a piecewise harmonic 
$\alpha$-fractal function $h^{\alpha}$ such that $h^{\alpha}\geq0$ for every $t\in SG$
and $\|f-h^{\alpha}\|_{\infty}<\epsilon.$
\end{theorem}
\begin{proof}

Let $\epsilon>0$ and $f \in SG$ with $f(t)\geq 0$. Using Theorem  (\ref{Denthm1}), one can choose a piecewise harmonic function $g$ such that 
$$\|f-g\|_{\infty}<\frac{\epsilon}{4}.$$
For $t\in SG$, define $h(t)=g(t)+\frac{\epsilon}{4}$. Then
$$h(t)=g(t)-f(t)+f(t)+\frac{\epsilon}{4}\geq -\|f-g\|_{\infty}+f(t)+\frac{\epsilon}{4}>f(t)\geq 0.$$
Furthermore,
$$\|f-h\|_{\infty}\leq \|f-g\|_{\infty}+\|g-h\|_{\infty}<\frac{\epsilon}{2}.$$
 Therefore, we obtain a piecewise harmonic function $h$ with $h(t)\geq 0$ and $\|f-h\|_{\infty}<\frac{\epsilon}{2}$. Theorem \ref{minmax} yields to get $h^{\alpha}$, a positive fractal perturbation of $h$, now we
choose $\alpha \in \mathcal{C}(SG)$  so that $\|\alpha\|_{\infty}\leq\frac{\epsilon}{\epsilon+2\|h-b\|_{\infty}}$. The estimate obtained from Proposition $\ref{ee}$ conveys the following:
\begin{equation}
    \begin{split}
\|f-h^{\alpha}\|_{\infty} & \leq  \|f-h\|_{\infty}+\|h-h^{\alpha}\|_{\infty}\\  
&\leq   \|f-h\|_{\infty}+\frac{\|\alpha\|_{\infty}}{1-{\|\alpha\|_{\infty}}}\|h-b\|_{\infty}\\ & < \frac{\epsilon}{2}+\frac{\epsilon}{2}\\ &= \epsilon.
    \end{split}
\end{equation}
This completes the proof.
\end{proof}
\begin{theorem}
Let $f\in \mathcal{C}(SG,\mathbb{R})$ and $\epsilon>0$. Then a fractal
piecewise harmonic function $h^{\alpha}$ can be obtained, which always lies above $f$, that is, $h^{\alpha}(t)\geq f(t)$, $\forall$ $t\in SG$ and $\|f-h^{\alpha}\|_{\infty} < \epsilon$.
\end{theorem} 
\begin{proof}
 For a given $f\in SG$ and $\epsilon>0$, it follows from Theorem \ref{Denthm1} that there exists a piecewise harmonic function $h$ with $h(t)\geq f(t)$ for every $t\in SG$ and $$\|f-h\|_{\infty}<\frac{\epsilon}{2}.$$
 So choosing $b(t)\leq h(t)$ and $\alpha \geq 0$ for every $t \in SG$ with $\|\alpha\|_{\infty}\leq\frac{\epsilon}{\epsilon+2\|h-b\|_{\infty}}$
 and in light of Theorem \ref{cc}, we obtain $h^{\alpha}(t)\geq h(t)  $ for every $t\in SG$. Moreover,
 $$\|f-h^{\alpha}\|_{\infty}<\epsilon.$$
  We, therefore, have a fractal polynomial $h^{\alpha}$ satisfying the required conditions.
\end{proof}
 \begin{theorem}
Let $f: SG \to \mathbb{R}$ be bounded below and integrable function on $SG$. Then we can choose fractal piecewise harmonic function of best one-sided approximant from below to $f$ on $SG$. 
\end{theorem}
\begin{proof}
Using a basic result of real analysis \cite{11jordan} there exists a sequence $(h_{m}^{\alpha})\in \mathcal{Y}_{f}^{\alpha}$ such that 
\begin{equation}\label{ww}
  \int_{SG}h_{m}^{\alpha}  \to A ~\text{as}~ m\to \infty  .   
\end{equation}

Now for some positive number $M_{SG}$, we obtain

 \begin{equation}
     \begin{aligned}
            \int_{SG}|h_{m}^{\alpha}|~d\mu_{p}~&= \int_{SG}|h_{m}^{\alpha}-A+A|~d\mu_{p}\\
            ~&\leq \int_{SG}|h_{m}^{\alpha}-A|~d\mu_{p} + \int_{SG}A ~d\mu_{p}.
     \end{aligned}
 \end{equation}

\par Recall that any $A\subset X$ in  finite dimensional normed linear space is closed and bounded iff $A$ is compact.
  \par 
  Since $\mathcal{Y}_{f}^{\alpha}$ is compact, we have a subsequence $(h_{m_k}^{\alpha})$ such that $(h_{m_k}^{\alpha})$ converges to $h^{\alpha}$ in $\mathcal{L}^{1}(SG)$. Remember that on a finite dimensional linear space, every norm is equivalent.

Since $\mathcal{H}_{n}^{\alpha}(SG) \big(\supseteq \mathcal{Y}_{f}^{\alpha}\big)$ is finite dimensional, it follows that the subsequence $(h_{m_k}^{\alpha})$ also converges to $h^{\alpha}$ uniformly. Since $h_{m}^{\alpha}(x)\leq f(x) \forall x \in SG$ and $h_{m_k}^{\alpha}\to h^{\alpha}$ uniformly, we get $h^{\alpha}(x)\leq f(x), \forall x\in SG$. Thus, 
$h^{\alpha} \in \mathcal{Y}_{f}^{\alpha}$. Using  (\ref{ww}), we have
$$\int_{SG} h^{\alpha} d\mu_{p} = \lim_{k\to \infty} \int_{SG} h_{m_k}^{\alpha} d\mu_{p} =A.$$
This completes the proof.

\end{proof}

\section{Best Approximation Property }
Recall that \cite[Corollary 3.3]{stri1} if $g$ is a non-constant function in the domain of $\Delta$, then $g^{2}$ is not in the domain of $\Delta$. 

\begin{proposition}\label{2.1pro}
(\cite{15pro}, P. 440). Let $Y$ be a normed linear space and $W$ is a
nonempty approximately compact subset of $Y$, then the metric projection $P_W:Y \to W$ is upper semicontinuous, and its values are compact.
\end{proposition}
\begin{proposition}\label{2.2pro}
(\cite{15pro}, P. 434). Let $Z$, $W$ be topological spaces and $W$ be
Hausdorff. If the set valued map $S: Z \to W$ is upper semicontinuous with compact values, then $S$ is closed, i.e., for each net ${z_{\lambda}}$ in $Z$, $z_{\lambda}\to z_{0}$, $w_{\lambda} \in S(z_{\lambda})$. 
\end{proposition}
\begin{theorem}
The following properties hold for the space $\mathcal{H}^{\alpha}_n(SG)$ of all $\alpha$- fractal polynomials, which is analogous to $n$ degree polynomial
    \begin{itemize}
    \item[(i)]
 $\mathcal{H} ^{\alpha}_n(SG)$ is a proximinal subset of $\mathcal{C}(SG)$. In fact, $\mathcal{H} ^{\alpha}_n(SG)$ is strongly proximinal. 
 \item[(ii)] For each $f \in SG$, the set of best approximants $P_{\mathcal{H} ^{\alpha}_n(SG)}(f)$ is closed and convex. In particular, $P_{\mathcal{H} ^{\alpha}_n(SG)}(f)$ is weakly closed.
 
\item[(iii)] The metric projection $P_{\mathcal{H} ^{\alpha}_n(SG)}:\mathcal{C}(SG) \to \mathcal{H} ^{\alpha}_n(SG) $ is upper
semicontinuous, closed, and locally bounded. 
 \item[(iv)]  $P_{\mathcal{H} ^{\alpha}_n(SG)}$ has a first Baire class selector $\mathcal{J}:\mathcal{C}(SG) \to \mathcal{H} ^{\alpha}_n(SG) $ and the set of points where $\mathcal{J}$  is norm-discontinuous is an $F_{\sigma}$- set of first category in $SG$.
    \end{itemize}
    \end{theorem} 
\begin{proof}
Firstly from $\mathcal{H} ^{\alpha}_n(SG)= \mathcal{F}^{\alpha}(\mathcal{H}_n(SG))$ and $\mathcal{F}^{\alpha}$ is a linear map, it is evident that $\mathcal{H} ^{\alpha}_n(SG)$ is spanned by a finite basis and therefore finite dimensional. Every subspace whose dimension is finite of a normed linear space are proximinal (see, for instance, \cite{15pro}). This leads to the conclusion that  $\mathcal{H} ^{\alpha}_n(SG)$ is proximinal. It is possible to demonstrate that a finite dimensional subspace of a normed linear space is strongly proximinal by applying the notion of compactness given in \cite{indu}, and thus, $\mathcal{H} ^{\alpha}_n(SG)$ is strongly proximinal. 
\\
Since  $\mathcal{H} ^{\alpha}_n(SG)$ is finite dimensional and Proximinal and therefore  $P_{\mathcal{H} ^{\alpha}_n(SG)}(f)$ is closed
and  $\mathcal{H} ^{\alpha}_n(SG)$ is convex.  It follows
that $P_{\mathcal{H} ^{\alpha}_n(SG)}(f)$ is convex for each $f$. Using the fact that $P_{\mathcal{H} ^{\alpha}_n(SG)}(f)$ is  closed and convex
subset of a normed linear space, we get  $P_{\mathcal{H} ^{\alpha}_n(SG)}(f)$ is weakly closed (see \cite{15pro}).
Proceeding to our next step, we show that $\mathcal{H} ^{\alpha}_n(SG)$ is approximately compact.  Let $f$ be a function in $SG$ and $h_{m}^{\alpha}\in \mathcal{H} ^{\alpha}_n(SG), m=1, 2, \dots$ be a minimizing sequence for $f$, i.e., $\|h_{m}^{\alpha}-f\|_{\infty} \to d(f, \mathcal{H} ^{\alpha}_n(SG) )$ as $m \to \infty$. Then one gets $M \in \mathbb{N}$ such that if $m \geq M$, then

 $$\|h_m^{\alpha}\|_{\infty} \le \|h_m^{\alpha}-f\|_{\infty}+\|f\|_{\infty} \leq 1+ d(f,\mathcal{H} ^{\alpha}_n(SG)) + \|f\|_{\infty} := A_{1}.$$

 Since $\mathcal{H} ^{\alpha}_n(SG)$ is Finite-dimensional, from the above-mentioned inequalities, it follows that $\mathcal{H} ^{\alpha}_n(SG)$ contains in compact set defined by the inequality\\ $\|g\|_{\infty} \leq A_{2} := \max\{\|h_1^{\alpha}\|_{\infty}, \|h_2^{\alpha}\|_{\infty},\dots,\|h_{M-1}^{\alpha}\|_{\infty}, A_1 \}. $ Hence, with the help of compact argument, one can choose a subsequence $h_{m_k}^{\alpha} \to p^{\alpha}$ for some $h^{\alpha}\in \mathcal{H} ^{\alpha}_n(SG)$ and hence  $\mathcal{H} ^{\alpha}_n(SG)$ approximatly compact. Therefore, Proposition \ref{2.1pro} allows us to deduce that the metric projection $P_{\mathcal{H} ^{\alpha}_n(SG)}$ is upper semicontinuous with  $P_{\mathcal{H} ^{\alpha}_n(SG)}(f)$ compact for each $f$.
Proposition \ref{2.2pro} concludes that $\mathcal{H} ^{\alpha}_n(SG)$ is Hausdorff, thus the multifunction $P_{\mathcal{H} ^{\alpha}_n(SG)}: \mathcal{C}(SG) \in \mathcal{H} ^{\alpha}_n(SG)$ is closed. It is simple to verify that a metric projection is locally bounded. For the completeness, we shall give some information. Let $f^{*} \in \mathcal{C}(SG)$. Consider the ball $B(f^{*}, r) $ centred at $f^{*}$, where $f \in B(f^{*}, r)$ and the radius $r>0$, and $h^{\alpha} \in P_{\mathcal{H} ^{\alpha}_n(SG)}(f)$. \\
 Then
 
 \begin{equation}
 \begin{aligned}
        \|h^{\alpha}\|_{\infty} &\leq \|h^{\alpha}-f\|_{\infty} + \|f-f^{*}\|+\|f^{*}\|_{\infty}\\
        &= d(f, \mathcal{H} ^{\alpha}_n(SG)) + \|f-f^{*}\|+\|f^{*}\|_{\infty}\\
        &\leq d(f^{*}, \mathcal{H} ^{\alpha}_n(SG)) + 2 \|f-f^{*}\|+\|f^{*}\|_{\infty}\\
        &\leq d(f^{*}, \mathcal{H} ^{\alpha}_n(SG)) + 2r +\|f^{*}\|_{\infty}\\
        &:=A_3,
        \end{aligned}    
 \end{equation}
 establishing the local boundedness of $ P_{\mathcal{H} ^{\alpha}_n(SG)}$.
 
Note that every finite dimensional normed linear space is reflexive and each
reflexive space satisfies the Radon-Nikodym property. Hence, according to item (3) the Jayne-Rogers Theorem ( \cite{11jay}, Theorem $7$) implies that   $ P_{\mathcal{H} ^{\alpha}_n(SG)}$ has a first Baire class selector $J$ whose set of points of discontinuity is an $F_{\sigma}$- set of first category in $\mathcal{C}(SG).$

\end{proof}
In the next section, we will point out that several distinct properties exist between the polynomials defined on any interval of the real line and polynomials defined on the $SG$.
  \section{Some remarks on the class of polynomials on SG}

  \subsection{Dimension:}
  Let $p_{i}: D \subset \mathbb{R} \to \mathbb{R}$ be the polynomials  of degree $\leq n_{i}$. Then $\dim(p_{i}) = \dim(D)$.
  \par In \big( \cite{vermabv}, Theorem 3.5\big), the following theorem  contradicts the above result on the space of the polynomials defined on the $SG$ also known as multi-harmonic space.

  \begin{theorem}
  If $f: SG \to \mathbb{R}$ is a continuous function and $E(f) < \infty $, then
 $$\frac{\log3}{\log2} \leq \dim_{H}(Gr(f))\leq \overline{\dim}_B(Gr(f)\leq \frac{\log(108/5)}{2\log2}.$$
  \end{theorem}

\subsection{Number of zeros:}   
The polynomial of degree $n$ defined on the interval has at most $n$ zeros. In the multi-harmonic spaces, also known as polynomials defined on the $SG$, can have infinite zeros.
\subsection{Bounded Variation}
Every polynomial defined on the compact interval is of bounded variation. But this property does not hold in the case of polynomials defined on $SG$. This follows from the 
 two interesting notions of bounded variation. One is demonstrated by Verma et al. in \big(\cite{vermabv}, Remark (4.13)\big). Recall that every non-constant harmonic function $h$ is not of bounded variation , and the other one is proposed by Ruiz et al. in  \big(\cite{Pat}, Theorem 5.2\big). Recall that on the $SG$ any non-constant piecewise harmonic function is not in bounded variation. 

\subsection{Multiplicative properties:}
Let  $p, q: [a,b] \to \mathbb{R}$ be two polynomials with degree $\leq n$, defined on the compact interval $[a,b]$, then the degree of $pq$ is  $\leq n^{2}$. Whereas, the 
 every non-constant  polynomial $p$ defined on $SG$,  i.e., harmonic function \big($p\in\mathcal{H}_{0}$\big) or say $p \in \text{dom} (\Delta)$, then $p^{2} \notin \text{dom} (\Delta)$; see, for instance, \big(\cite{stri1}, Corollary 3.3\big).
Despite this major disadvantage, we can overcome it by defining a different Laplacian with a different measure. Kusuoka \cite{K1, K2} defined a measure for
which the domain of its Laplacian is closed under multiplication. However, the
Kusuoka measure is not self-similar and that makes it significantly more difficult to study.

\section{Dimensional Results}
In this section, we provide some bounds for the dimension of the graph of an $\alpha$-fractal function.
\begin{definition}
Consider a closed and bounded subset $X \subset \mathbb{R}^2$. We denote the oscillation of $g: X \to \mathbb{R}$ over the  $X$ by $\text{OSC}_{g} [X]$ and define as 
$$ \text{OSC}_{g} [X]=\sup _{t,z \in X} | g(t)-g(z) |$$
with
\end{definition}
$$\alpha_{\max}=\max \{\|\alpha_{\boldsymbol{\omega}}\|_{\infty}: \boldsymbol{\omega} \in I^N\},$$
   $$\alpha_{\min}=\min \{\|\alpha_{\boldsymbol{\omega}}\|_{\infty}:\boldsymbol{\omega} \in I^N\}.$$
\begin{lemma}\label{cubes}
Suppose that $f \in \mathcal{C}(SG)$. Let $\delta=\frac{1}{2^{n}}$ for some $n\in \mathbb{N}$ and  the number of $\delta$-cubes intersecting the graph of $f$, denoted by $N_{\delta}(Gr(f))$, then
$$2^{n}\sum_{\omega\in I^n}\text{OSC}_{f}[L_{\omega}{(SG)}]\leq N_{\delta}(Gr(f))\leq 2.3^{n}+2^{n}\sum_{\omega\in I^n}\text{OSC}_{f}[L_{\omega}{(SG)}].$$
\end{lemma}
\begin{lemma}
Consider $f^{\alpha}$ been constructed using the original functions $f$, $\alpha_{\boldsymbol{\omega}}$, $b$. If $f^{\alpha}$ preserves the H\"{o}lder continuity with exponents $\sigma_{f}, \sigma_{\alpha}, \sigma_{b}$ and H\"{o}lder constants $K_{f}, K_{\alpha}, K_b$ respectively, then
  \begin{equation}
\begin{aligned}
\text{OSC}_{F_{{\omega}_{m}}}\big(G_{{\omega}_{1}{\omega}_{2}\dots {\omega}_{m-1}} \big) \leq & \big\|\alpha_{{\omega}_{m}}\big\|_{\infty}~\text{OSC}_{f^{\alpha}}\big(L_{{\omega}_{1}{\omega}_{2}\dots {\omega}_{m-1}} \big) \\
~&+ \frac{K_b \big\|\alpha_{{\omega}_{m}}\big\|_{\infty}+K_{\alpha}\ \big(\big\|b\big\|_{\infty} + \big\|f^{\alpha}\big\|_{\infty}  \big)}{2^{N\sigma(m-1)}}+  \frac{K_{f}}{2^{N\sigma}}.
\end{aligned}
\end{equation}
\end{lemma} 

\begin{proof}
Note that 
$$F_{\boldsymbol{\omega}_1}(t, x)= \alpha_{\boldsymbol{\omega}_1}(t)x+f(L_{\boldsymbol{\omega}_1}(t))-\alpha_{\boldsymbol{\omega}_1}(t)b(t).$$
Taking $t, u \in SG$, one obtains

\begin{equation}
    \begin{aligned}
           |F_{\omega_{1}}(t, f^{\alpha}(t))-F_{\omega_{1}}(u, f^{\alpha}(u))| 
           =&~ \big| \alpha_{\omega_{1}}(t)f^{\alpha}(t) + f(L_{\omega_{1}}(t))- \alpha_{\omega_{1}}(t)b(t)\\
       &-    \alpha_{\omega_{1}}(u)f^{\alpha}(u) - f(L_{\omega_{1}}(u))+ \alpha_{\omega_{1}}(u)b(u) \big|\\
           \leq&~ \big|\alpha_{\omega_{1}}(t)\big|~ \big|f^{\alpha}(t)-f^{\alpha}(u)\big|+  \big|\alpha_{\omega_{1}}(t)\big|~\big|b(t)-b(u)\big|\\
          &+ (\big|b(u)| + |f^{\alpha}(u)\big|)~ \big|\big(\alpha_{\omega_{1}}(t)-\alpha_{{\omega}_{1}}(u)\big)\big|\\
          &+ \big|f(L_{{\omega}_{1}}(t))- f(L_{{\omega}_{1}}(u))\big|\\
          \leq &~\big \|\alpha_{{\omega}_{1}}\big \|_{\infty} \big|f^{\alpha}(t)-f^{\alpha}(u)\big|+ K_b \big \|\alpha_{{\omega}_{1}}\big \|_{\infty} \big\|t-u\big \|_{2}^{\sigma_{b}}\\
          &+ K_{f} \big\|L_{{\omega}_1}(t)-L_{{\omega}_1}(u)\big\|\\
           \leq &~\big \|\alpha_{{\omega}_{1}}\big \|_{\infty} \text{OSC}_{f^{\alpha}}(SG) + \big\|\alpha_{{\omega}_{1}}\big\|_{\infty} K_b+K_{\alpha}\ \big(\big\|b\big\|_{\infty} + \big\|f^{\alpha}\big\|_{\infty}  \big)\\
           &+ \frac{K_{f}}{2^{N\sigma}}.
          \end{aligned}
         \end{equation}
         
Performing iteration, we have
\begin{equation}
\begin{aligned}
\text{OSC}_{F_{{\omega}_{m}}}\big(G_{{\omega}_{1}{\omega}_{2}\dots {\omega}_{m-1}} \big) \leq & \big\|\alpha_{{\omega}_{m}}\big\|_{\infty}~\text{OSC}_{f^{\alpha}}\big(L_{{\omega}_{1}{\omega}_{2}\dots {\omega}_{m-1}} \big) \\
~&+ \frac{K_b \big\|\alpha_{{\omega}_{m}}\big\|_{\infty}+K_{\alpha}\ \big(\big\|b\big\|_{\infty} + \big\|f^{\alpha}\big\|_{\infty}  \big)}{2^{N\sigma(m-1)}}+  \frac{K_{f}}{2^{N\sigma}}.
\end{aligned}
\end{equation}
This completes the proof.
\end{proof}

	\begin{theorem}
Let $f^{\alpha}$ be $\alpha$- fractal function on $SG$. Then we have the following:
\begin{itemize}
    \item[(1).] If $\alpha_{\min}>\frac{1}{2^{nN}}$, then $\frac{\log3}{\log2}\leq \dim_{H}(Gr(f^{\alpha}))  \leq ~\underline{\dim_{B}}(Gr(f^{\alpha}))\\  \leq   
    \overline{\dim_{B}}(Gr(f^{\alpha})) \leq~ 1+ \frac{\log(\tilde{\gamma})}{N \log2}.$
    \item[(2).] If $\frac{1}{2^{nN}} < \alpha_{max} < 1$, then $\frac{\log3}{\log2}\leq \dim_{H}(Gr(f^{\alpha}))  \leq ~\underline{\dim_{B}}(Gr(f^{\alpha}))\\  \leq   
    \overline{\dim_{B}}(Gr(f^{\alpha})) \leq~ 1 + \frac{\log(3 \alpha_{\max})}{N\log2}.$
\end{itemize}

\end{theorem}

\begin{proof}

Note that $f^{\alpha}$ holds the following  self-referential equation:
$$f^{\alpha}(t)= F_{\boldsymbol{\omega}}\big(L_{\boldsymbol{\omega}}^{-1}(t), f^{\alpha}(t)\big)
~~\forall~ t\in L_{\boldsymbol{\omega}}(SG) ~\text ~{and}~\boldsymbol{\omega} \in I^{N}.$$

Let $t, u \in L_{{\omega}_{1}}(SG),$ we have

\begin{equation}
\begin{aligned}
|f^{\alpha}(t)-f^{\alpha}(u)|&= |F_{{\omega}_{1}}\big(L_{{\omega}_{1}}^{-1}(t), f^{\alpha}(t)\big)-F_{{\omega}_{1}}\big(L_{{\omega}_{1}}^{-1}(u), f^{\alpha}(u)\big)|\\
&\leq \sup_{\tilde{t},\tilde{u}\in L_{{\omega}_{1}}(SG)} |F_{{\omega}_{1}}\big(L_{{\omega}_{1}}^{-1}(\tilde{t}), f^{\alpha}(\tilde{t})\big)-F_{{\omega}_{1}}\big(L_{{\omega}_{1}}^{-1}(\tilde{u}), f^{\alpha}(\tilde{u})\big)|\\
~&~= \text{OSC}_{F_{{\omega}_{1}}}\big[ Gr(f^{\alpha})\big].
\end{aligned}
\end{equation}
Hence, for any $t, u \in L_{{\omega}_{1}}(SG)$, 

$$\text{OSC}_{f^{\alpha}}\big[ L_{{\omega}_{1}{\omega}_{2}\dots {\omega}_{m
}} \big(SG)] \leq \text{OSC}_{F_{{\omega}_{m}}} \big[ G_{{\omega}_{1}{\omega}_{2}\dots {\omega}_{m-1}} \big].$$

Using the previous lemma, one obtains
\begin{equation}
\begin{aligned}
\text{OSC}_{f^{\alpha}}&\big[ L_{{\omega}_{1}{\omega}_{2}\cdots {\omega}_{m
}} \big(SG)] \\~\leq &~ \|\alpha_{{\omega}_{m}}\|_{\infty} \|\alpha_{{\omega}_{m-1}}\|_{\infty}\dots \|\alpha_{{\omega}_{1}}\|_{\infty} \text{OSC}_{f^{\alpha}}[SG]\\
&+ K_b \bigg[\frac{\|\alpha_{{\omega}_{m}}\|_{\infty}}{2^{(m-1)N\sigma}}   +
\frac{\|\alpha_{{\omega}_{m}}\|_{\infty}\|\alpha_{{\omega}_{m-1}}\|_{\infty}}{2^{(m-2)N\sigma}}
+ \cdots +
\|\alpha_{{\omega}_{m}}\|_{\infty}\|\alpha_{{\omega}_{m-1}}\|_{\infty}\dots\|\alpha_{{\omega}_{1}}\|_{\infty}   
   \bigg]\\
   &+  K_{f} \bigg[ \frac{1}{2^{mN\sigma}}
   + \frac{\|\alpha_{{\omega}_{m}}\|_{\infty}}{2^{(m-1)N\sigma}}
   +\frac{\|\alpha_{{\omega}_{m}}\|_{\infty}\|\alpha_{{\omega}_{m-1}}\|_{\infty}}{2^{(m-2)N\sigma}}\\
   &+\dots 
   +\frac{\|\alpha_{{\omega}_{m}}\|_{\infty}\|\alpha_{{\omega}_{m-1}}\|_{\infty}\dots\|\alpha_{{\omega}_{2}}\|_{\infty}}{2^{N\sigma}} \bigg]\\
   &+ K_{\alpha}\big(\|b\|_{\infty}+\|f^{\alpha} \|_{\infty} \big) \bigg[ \frac{1}{2^{(m-1)N\sigma}}
      + \frac{\|\alpha_{{\omega}_{m}}\|_{\infty}}{2^{(m-2)N\sigma}}
      +\frac{\|\alpha_{{\omega}_{m}}\|_{\infty}\|\alpha_{{\omega}_{m-1}}\|_{\infty}}{2^{(m-3)N\sigma}}\\
      &+\dots 
      +\big(\|\alpha_{{\omega}_{m}}\|_{\infty}\|\alpha_{{\omega}_{m-1}}\|_{\infty}\dots\|\alpha_{{\omega}_{2}}\|_{\infty}\big) \bigg].
   \end{aligned}
\end{equation}
Now, we are well-equipped to estimate the fractal dimension of $Gr(f^{\alpha})$.
\begin{itemize}
       \item[(1).] If $\alpha_{min} >\frac{1}{2^{N\sigma}} $, then we get
       \begin{equation*}
       \begin{aligned}
       \text{OSC}_{f^{\alpha}}&\big[ L_{{\omega}_{1}{\omega}_{2}\dots {\omega}_{m
       }} \big(SG)] \\ \leq &~ \|\alpha_{{\omega}_{m}}\|_{\infty} \|\alpha_{{\omega}_{m-1}}\|_{\infty}\dots \|\alpha_{{\omega}_{1}}\|_{\infty} \text{OSC}_{f^{\alpha}}[SG]\\
       &+ K_b\|\alpha_{{\omega}_{m}}\|_{\infty} \|\alpha_{{\omega}_{m-1}}\|_{\infty}\dots \|\alpha_{{\omega}_{1}}\|_{\infty}\bigg[1+\frac{1}{\alpha_{\min} 2^{N\sigma}}+\dots + \frac{1}{\alpha_{\min}^{m}2^{mN\sigma}}              \bigg]\\
       &+ \frac{K_{f}\|\alpha_{{\omega}_{m}}\|_{\infty} \|\alpha_{{\omega}_{m-1}}\|_{\infty}\dots \|\alpha_{{\omega}_{2}}\|_{\infty}}{2^{N\sigma}} \bigg[1+\frac{1}{\alpha_{\min} 2^{N\sigma}}+\dots +\frac{1}{\alpha_{\min}^{m-1}2^{(m-1)N\sigma}}     \bigg]\\
       &+ K_{\alpha}\big(\|b\|_{\infty}+\|f^{\alpha} \|_{\infty} \big)\|\alpha_{{\omega}_{m}}\|_{\infty} \|\alpha_{{\omega}_{m-1}}\|_{\infty}\dots \|\alpha_{{\omega}_{2}}\|_{\infty}\bigg[1+\frac{1}{\alpha_{\min} 2^{N\sigma}}+\dots\\~&\dots +\frac{1}{\alpha_{\min}^{m-1}2^{(m-1)N\sigma}}     \bigg]\\
       \leq &~  \|\alpha_{{\omega}_{m}}\|_{\infty} \|\alpha_{{\omega}_{m-1}}\|_{\infty}\dots \|\alpha_{{\omega}_{1}}\|_{\infty}\bigg[\text{OSC}_{f^{\alpha}}[SG]\\~& + \bigg(K_b+\frac{K_{f}}{\alpha_{\min}2^{N\sigma}}+\frac{K_{\alpha}\big(\|b\|_{\infty}+\|f^{\alpha} \|_{\infty} \big)}{\alpha_{\min}}\bigg)\bigg( \frac{1}{1-\frac{1}{\alpha_{\min}2^{N\sigma}}}\bigg)    \bigg].
      \end{aligned} 
       \end{equation*}
       Finally, we have
       $$ \text{OSC}_{f^{\alpha}}\big[ L_{{\omega}_{1}{\omega}_{2}\dots {\omega}_{m
              }} \big(SG)] \leq R_{1} \|\alpha_{{\omega}_{m}}\|_{\infty} \|\alpha_{{\omega}_{m-1}}\|_{\infty}\dots \|\alpha_{{\omega}_{1}}\|_{\infty} ,$$
              where $R_{1}=\text{OSC}_{f^{\alpha}}[SG] + \bigg(K_b+\frac{K_{f}}{\alpha_{\min}2^{N\sigma}}+\frac{K_{\alpha}\big(\|b\|_{\infty}+\|f^{\alpha} \|_{\infty} \big)}{\alpha_{\min}}\bigg)\bigg( \frac{1}{1-\frac{1}{\alpha_{\min}2^{N\sigma}}}\bigg).$
              Lemma (\ref{cubes}) now yields 
              \begin{equation}
              \begin{aligned}
              N_{\delta}(Gr(f^{\alpha}))~&\leq 2.3^{n}+2^{nN}\sum_{\omega\in I^n}\text{OSC}_{f^{\alpha}}\big[ L_{{\omega}_{1}{\omega}_{2}\dots {\omega}_{m
                     }} \big(SG)]\\
              ~&\leq 2.3^{n}+ 2^{nN}\sum_{\omega\in I^n} R_{1} \|\alpha_{{\omega}_{m}}\|_{\infty} \|\alpha_{{\omega}_{m-1}}\|_{\infty}\dots \|\alpha_{{\omega}_{1}}\|_{\infty}\\
              &~ \leq 2.3^{n}+ 2^{nN}R_{1} \bar{\gamma}^{n},
               \end{aligned}
              \end{equation}
      where $\bar{\gamma}= \sum_{\omega\in I^n} \|\alpha_{\omega}\|_{\infty}$. Further, we obtain
     \begin{equation}
     \begin{aligned}
    \overline{\lim}_{\delta\to 0} \frac{\log( N_{\delta}(Gr(f^{\alpha})))}{-\log(\delta)}~&\leq \overline{\lim}_{n\to{\infty}}\frac{ \log(2.3^{n}+ 2^{nN}R_{1} \bar{\gamma}^{n})}{nN\log(2)}\\
    &~= \overline{\lim}_{n\to{\infty}}  \frac{\log(2^{nN}R_{1}\bar{\gamma}^{n})+\log\bigg(1+\frac{2.3^{n}}{2^{nN}R_{1}\bar{\gamma}^{n}}    \bigg) }{nN\log2}.
     \end{aligned}
     \end{equation}
   Since $\bar{\gamma}= \sum_{\omega\in I^n} \|\alpha_{\omega}\|_{\infty}$, we have $\tilde{\gamma}\geq 3^{n}\alpha_{\min}$.  Combining the inequality $\tilde{\gamma}\geq 3^{n}\alpha_{\min}$ and $N>N\sigma$ with the above inequality (i.e. $\alpha_{min} >\frac{1}{2^{N\sigma}} $). This gives $\frac{3}{\tilde{\gamma}2^{N}}<1$. Therefore, we obtain
     $$\overline{\lim}_{\delta\to 0} \frac{\log( N_{\delta}(Gr(f^{\alpha})))}{-\log(\delta)} \leq 1+ \frac{\log(\bar{\gamma})}{N \log2}.$$

       \item[(2).] If $\frac{1}{2^{N\sigma}}<\alpha_{\max}<1$, then we get
       \begin{equation*}
              \begin{aligned}
               \text{OSC}_{f^{\alpha}}&\big[ L_{{\omega}_{1}{\omega}_{2}\dots {\omega}_{m
              }} \big(SG)] \\ \leq ~& \alpha_{\max}^{m} \text{OSC}_{f^{\alpha}}[SG]\\
              &+ K_b \alpha_{\max}^{m}\bigg[1+\frac{1}{\alpha_{\max} 2^{N\sigma}}+\cdots+ \frac{1}{\alpha_{\max}^{m}2^{mN\sigma}}              \bigg]\\
              &+ \frac{K_{f}\alpha_{\max}^{m-1}}{2^{N\sigma}} \bigg[1+\frac{1}{\alpha_{\max} 2^{N\sigma}}+\cdots+ \frac{1}{\alpha_{\max}^{m-1}2^{(m-1)N\sigma}}     \bigg]\\
              &+ K_{\alpha}\big(\|b\|_{\infty}+\|f^{\alpha} \|_{\infty} \big)\alpha_{\max}^{m-1}\bigg[1+\frac{1}{\alpha_{\max} 2^{N\sigma}}+\dots +\frac{1}{\alpha_{\max}^{m-1}2^{(m-1)N\sigma}}     \bigg]\\
               \leq &~ \alpha_{max}^{m}\bigg[\text{OSC}_{f^{\alpha}}[SG] + \bigg(K_b+\frac{K_{f}}{\alpha_{\max}2^{N\sigma}}+\frac{K_{\alpha}\big(\|b\|_{\infty}+\|f^{\alpha} \|_{\infty} \big)}{\alpha_{\max}}\bigg)\bigg( \frac{1}{1-\frac{1}{\alpha_{\max}2^{N\sigma}}}\bigg)    \bigg].
             \end{aligned} 
              \end{equation*}
              Hence,
              $$ \text{OSC}_{f^{\alpha}}\big[ L_{{\omega}_{1}{\omega}_{2}\dots {\omega}_{m
                     }} \big(SG)] \leq R_{2} \alpha_{max}^{m}, $$
                     where $R_{2}=\text{OSC}_{f^{\alpha}}[SG] + \bigg(K_b+\frac{K_{f}}{\alpha_{\max}2^{N\sigma}}+\frac{K_{\alpha}\big(\|b\|_{\infty}+\|f^{\alpha} \|_{\infty} \big)}{\alpha_{\max}}\bigg)\bigg( \frac{1}{1-\frac{1}{\alpha_{\max}2^{N\sigma}}}\bigg).$
                     Lemma (\ref{cubes}) yields 
                     \begin{equation}
                     \begin{aligned}
                     N_{\delta}(Gr(f^{\alpha}))~&\leq 2.3^{n}+2^{n}\sum_{\omega\in I^n}\text{OSC}_{f^{\alpha}}\big[ L_{{\omega}_{1}{\omega}_{2}\dots {\omega}_{m
                     }} \big(SG)]\\
                     ~&\leq 2.3^{n}+ 2^{n}.3^{n} R_{2} \alpha_{\max}^{n}.
                      \end{aligned}
                     \end{equation}
                      Further, we obtain
                          \begin{equation}
                          \begin{aligned}
                         \overline{\lim}_{\delta\to 0} \frac{\log( N_{\delta}(Gr(f^{\alpha})))}{-\log(\delta)}
                 \leq&~ \overline{\lim}_{n\to{\infty}}\frac{ \log(2.3^{n}+2^{nN}.3^{n} R_{2} \alpha_{\max}^{n})}{nN\log(2)}\\
                 = &~  \overline{\lim}_{n\to{\infty}} \frac{\log(2^{nN}.3^{n}R_{2}\alpha_{\max}^{n})+\log\bigg(1+\frac{2.3^{n}}{2^{nN}.3^{n}R_{2}\alpha_{\max}^{n}}\bigg)}{nN\log2}. 
                         \end{aligned}
                          \end{equation}
   \end{itemize}
   
 Plugging the inequality $2^{N}> 2^{N \sigma}$ with the aforementioned inequality (i.e.  $\frac{1}{2^{N\sigma}}<\alpha_{\max}<1) $), we get
 $$\overline{\lim}_{\delta\to 0} \frac{\log( N_{\delta}(Gr(f^{\alpha})))}{-\log(\delta)}
                 \leq1 + \frac{\log(3 \alpha_{\max})}{N\log2}.$$
Hence, we complete the proof.
\end{proof}

\section{Conclusion and future remarks}
In this article, by providing sufficient conditions to the parameters, we can ensure that the perturbation function $f^{\alpha}$  exhibits certain properties of the original function $f$. Furthermore, we have addressed the existence of a fractal piecewise harmonic function of the best one-sided approximant from below to $f$ on $SG$. Under certain conditions, we have also computed the bounds for $\overline{\dim}_{B}(Gr(f^{\alpha}))$. Moreover, we have presented interesting results distinguishing the fundamental properties of the polynomials defined on $SG$ and polynomials defined on an interval. Our future research may include investigating the dimensional results and other significant properties of the class of rational polynomials defined on the $SG$.

    \bibliographystyle{amsplain}

\begin{thebibliography}{10}
\bibitem{AGN} V. Agrawal, T. Som, S. Verma, On bivariate fractal approximation, J. Anal (2022) 1-19. 
\bibitem{EPJS} V. Agrawal, T. Som, Fractal dimension of $\alpha$-fractal function on the Sierpi\'nski Gasket, Eur. Phys. J. Spec. Top. 230 (21) (2021) 3781-3787. 
\bibitem{LP} V. Agrawal, T. Som, $\mathcal{L}^p$-approximation using fractal functions on the Sierpi\'nski Gasket, Results Math 77 (2) (2021) 1-17.
\bibitem{Pat} P. Alonso-Ruiz, F. Baudoin, L. Chen, L. Rogers, N. Shanmugalingam, A. Teplyaev, Besov class via heat semigroup on Dirichlet spaces III: BV functions and sub-Gaussian heat kernel estimates, Calc. Var. Partial Differential Equations 60 (2021), no. 5, Paper No. 170, 38 pp.
\bibitem{MF2} M. F. Barnsley, Fractals Everywhere, Academic Press, Orlando, Florida, 1988.
	\bibitem{CC} P. G. Casazza, O. Christensen, Perturbation of operators and application to
	frame theory, J. Fourier Anal. Appl. 3(5) (1997) 543-557.
	
\bibitem{SS2} S. Chandra, S. Abbas, The calculus of bivariate fractal interpolation surfaces, Fractals 29(3) (2020) 2150066.
\bibitem{Chana}   S. Chandra, S. Abbas, Analysis of fractal dimension of mixed Riemann-Liouville integral. Numerical Algorithms (2022) 1-26. 
\bibitem {Celik} D. Celik, S. Kocak, Y. \"Ozdemir, Fractal interpolation on the Sierpi\'nski Gasket, J. Math. Anal. Appl. 337 (2008) 343-347.
\bibitem {Fal} K. J. Falconer, Fractal Geometry: Mathematical Foundations and Applications, John Wiley Sons Inc., New York, 1999.
\bibitem{Gme} A. Gowrisankar, R. Uthayakumar, Fractional Calculus on Fractal Interpolation for a Sequence of Data with Countable Iterated Function System, Mediterr. J. Math. 13 (2016) 3887-3906 .


\bibitem{11jordan} R. A. Gordon, Real Analysis: A First Course, 2nd edition, Boston, Pearson Education Inc.,
2002. 


\bibitem{Hut} J. E. Hutchinson, Fractals and self similarity, Indiana Uni. Math. J. 30(5) (1981) 713-747.
	\bibitem{indu} V. Indumathi, Semi-continuity properties of metric projection. In: Nonlinear Analysis: Approximation Theory, Optimization and Applications (Q. H. Ansari, ed.), Chapter 2.
Birkhauser, New Delhi, India, pp. 33–59 (2014).

\bibitem{11jay} J. E. Jayne and C. A. Rogers, Borel selectors for upper semicontinuous set-valued
maps, Acta Math. 155(1) (1985) 41–79.


 \bibitem {Kig} J. Kigami, Analysis on Fractals, Cambridge University Press, Cambridge
(2001).
\bibitem{K1} S. Kusuoka, A diffusion process on a fractal, in ``Probabilistic Methods in Mathematical Physics (Katata/Kyoto, 1985),"  251-274, Academic Press, Boston, $1987 .$
\bibitem{K2} S. Kusuoka, Dirichlet forms on fractals and products of random matrices, Publ. Res. Inst. Math. Sci. 25 (1989) 659-680.
	\bibitem {Liang1} Y. S. Liang, Box dimensions of Riemann-Liouville fractional integrals of continuous functions of bounded variation, Nonlin. Anal. 72 (2010) 4304-4306.
\bibitem{15pro} H. N. Mhaskar and D. V. Pai, Fundamentals of Approximation Theory. Narosa, 2007.
New Delhi, India.

\bibitem {M2} M. A. Navascu\'es, Fractal polynomial interpolation, Z. Anal. Anwend. 25(2) (2005) 401-418.
 
 
\bibitem{Pr1} A. Priyadarshi, Lower bound on the Hausdorff dimension of a set of complex continued fractions. J. Math. Anal. Appl. 449(1) (2017) 91-95.
\bibitem{Ri2} S.-I.  Ri, Fractal functions on the Sierpi\'nski gasket, Chaos, Solitons and Fractals 138 (2020) 110142.



\bibitem {Ruan4} S.-G. Ri, H.-J. Ruan, Some properties of fractal interpolation functions on Sierpinski gasket, J. Math. Anal. Appl. 380 (2011) 313-322.

\bibitem {Ruan3} H.-J. Ruan, Fractal interpolation functions on post critically finite self-similar sets, Fractals 18 (2010) 119-125.







\bibitem{SP} A. Sahu, A. Priyadarshi, On the box-counting dimension of graphs of harmonic functions on the Sierpi\'{n}ski gasket, J. Math. Anal. Appl. 487 (2020) 124036. 	

\bibitem{stri} R. S. Strichartz, Differential Equations on Fractals, Princeton University Press, Princeton, NJ, 2006.
\bibitem{stri1} R. S. Strichartz, A. Teplyaev, What Is Not in the Domain of the Laplacian on Sierpi\'{n}ski Gasket Type Fractals, Journal of Functional Analysis 166 (1999) 197-217. 
\bibitem{Verma21} S. Jha, S. Verma, Dimensional Analysis of $\alpha$-Fractal Functions, Results Math 76 (4) (2021) 1-24.


\bibitem{vermabv} S. Verma, A. Sahu, Bounded Variation on the Sierpi\'nski Gasket, Fractals, doi: 10.1142/S0218348X2250147X, (2022).







\end{thebibliography}

\end{document}